\documentclass{elsarticle}

\usepackage{JBpack}
\usepackage{mathtools}
\DeclarePairedDelimiter{\ceil}{\lceil}{\rceil}

\usepackage{booktabs}
\usepackage{algorithm, algpseudocode}

\def\Hop{\overline{H_{y,Y^K}}}
\def\pinm{\pi_{n,m}}

\journal{Journal of Computational and Applied Mathematics}

\begin{document}

\begin{frontmatter}

\title{\emph{A posteriori} Error Estimation for the Spectral Deferred Correction Method}

\author{Jehanzeb H. Chaudhry\fnref{Zeb}}
\author{J.B. Collins\corref{cor}\fnref{Jeb}}

\cortext[cor]{Corresponding Author}
\fntext[Zeb]{Department of Mathematics and Statistics, The University of New Mexico, Albuquerque, NM 87131, United States}
\fntext[Jeb]{Department of Mathematics, University of Mary Washington, Fredericksburg, VA 22401, United States}

\ead{jehanzeb@unm.edu (J.H. Chaudhry), jcollin2@umw.edu}

\begin{abstract}
    The spectral deferred correction method is a variant of the deferred correction method for solving ordinary differential equations.  A benefit of this method is that is uses low order schemes iteratively to produce a high order approximation.  In this paper we consider adjoint-based \emph{a posteriori}  analysis to estimate  the error in a quantity of interest of the solution.  This error formula is derived by first developing a nodally equivalent finite element method to the spectral deferred correction method.  The error formula is then split into various terms, each of which characterizes a different component of the error.  These components may be used to determine the optimal strategy for changing the method parameters to best improve the error.
\end{abstract}

\end{frontmatter}

\section{Introduction}
\label{sec:introduction}

In this article we  consider \emph{a posteriori} error analysis for the spectral deferred correction (SDC) method applied to a system of  nonlinear ordinary differential equations (ODEs) of the form,
\begin{equation}
\begin{cases}
\dot{y}(t) = f(y(t),t), \quad t \in (0,T],\\
y(0) = y_0.
\end{cases}
\label{eq:ODE}
\end{equation}
Here $y(t), y_0 \in \Rbb^d$, $\dot{y}  = d{y}/ d{t}$ and $f:\Rbb^d \by \Rbb \rightarrow \Rbb^d$.  The goal of the \emph{a posteriori} analysis is to derive accurate error estimates in the numerically computed value of a bounded linear functional of the solution, the so called  quantity-of-interest (QoI),

\begin{equation}
\label{eq:qoi_definition}
	Q(y) =  \< y,\psi \>_{[0,T]} + (y(T), \psi_T), 
\end{equation}
where $(\cdot, \cdot)$ denotes the standard Euclidean inner product on $\Rbb^d$,  $\< \cdot, \cdot \>_{\omega}  = \int_{\omega} (\cdot, \cdot) \, dt$ denotes the standard $L^2(\omega;\Rbb^d)$ inner product on a time domain $\omega$,  $\psi \in L^2([0,T];\Rbb^d)$ and $\psi_T \in \Rbb^d$.  
Examples include $\psi(t) = 1/T$ corresponding to  a time average for $d=1$, or $\psi_T = [1/d, 1/d, \ldots, 1/d]^T$ corresponding to the average of the components of $y$ at the final time.

Spectral deferred correction (SDC) methods are variants of the traditional defect correction methods~\cite{Pereyra1967,Bohmer1984} and were introduced in~\cite{dutt:2000}. \emph{A priori} analysis and numerical studies of the method and its modifications have been carried out in~\cite{dutt:2000,Minion2003,Minion2004,Layton2004,Layton2008,Hansen2011}. All deferred correction methods share the basic idea of iteratively arriving at solutions of higher accuracy. In particular,  a time interval is divided into  smaller subintervals and  a lower order method is used on these subintervals to approximate the error, which is then added back to the approximate solution, and hence higher order accuracy is achieved. Traditional methods of approximating the error equation using differential operators are often unstable~\cite{Trefethen1987, Higham1993, dutt:2000}. In SDC methods, stability is achieved by replacing the differential operators by a corresponding Picard integral equation, which is then approximated on appropriately chosen subintervals based on Gaussian quadrature nodes~\cite{dutt:2000}. 
SDC methods are able to achieve high accuracy at a  computational cost comparable to that of the Runge-Kutta methods~\cite{Minion2010}.
SDC methods have been applied to a number of important physical problems and partial differential equations (PDEs)~\cite{Bourlioux2003, Layton2004, Minion2004, Liu2008}. SDC methods have also seen recent interest for their use in developing parallel-in-time methods~\cite{Minion2010,Emmett2012}.

The range of scales and the complexity of processes involved in practical scientific and engineering applications means that numerical simulations often have  a significant error~\cite{Estep:Larson:00,Fish:2009}. Reliable use of numerical simulations to make model predictions in science and engineering applications necessitates that this error in a QoI be quantified. Error estimation is also essential for constructing adaptive algorithms which identify the correct discretization parameters to be refined to arrive at more accurate solutions. 

\color{black}
In this paper, we develop an \emph{a posteriori} error analysis for SDC methods using variational analysis, computable residuals, and adjoint problems. Such \emph{a posteriori} error estimates are widely used for a wide range of numerical methods including finite element methods \cite{Estep:Larson:00,estep2,eehj1, AO2000, rannacherbook, Barth2004, BeckerRannacher-2001, GilesSuli-2002,Chaudhry17}, numerous time integrators like the IMEX methods, backward differentiation methods, Runge-Kutta, Lax-Wendroff,  multirate methods~\cite{CEG+2015,Chaudhry2019,collins:2015,collins:2014, Logg2004,EJL04}, parallel-in-time and domain decomposition methods~\cite{JESVS15,Chaudhry2019b}.

A recent \textit{a priori} analysis of the SDC method further indicates the various sources of error in the method~\cite{Causley2019}. These sources are identified as the error at the current time point, the error from the previous iterate, and the numerical integration error that comes from the total number of quadrature nodes used for integration. The resulting \textit{ a posteriori} error  estimates in this article also have the useful feature that the total error is decomposed as a sum of contributions from various aspects of discretization and therefore can provide insight into the effect of different choices for the parameters controlling the discretization.  We demonstrate how these error components can be used to determine effective adaptive error control of the numerical parameters of the method.
\color{black}
The paper is organized in the following manner.  In Section \ref{sec:SDC}, we introduce the SDC method, and develop the nodally equivalent finite element method.  We also look at how the order of the finite element method is chosen.  Section \ref{sec:err} contains our main error representation formula, introduces the various error contributions, and discusses how nonlinear problems are handled.  Finally, in Section \ref{sec:results}, we present results using the error representation formula, demonstrating the meaning of the various components of the error.  We examine both linear and nonlinear examples, and also look at how the order of the finite element method is important to computing an accurate error.

\section{The SDC Method and its Nodally Equivalent Finite Element Method}\label{sec:SDC}

This section introduces the SDC method, along with all relevant notation.  In order to perform adjoint-based error estimation, we require a globally defined function to represent the approximate solution.  Therefore, the rest of this section is devoted to developing and analyzing a finite element method whose approximation is equivalent in a certain sense to the SDC approximation.  We examine this finite element method by first looking at the convergence of the method, then examine the optimal order of convergence to match with the corresponding SDC approximation.  Finally, we look at how to compute the higher order finite element approximation using only the SDC approximation.

	\subsection{SDC Method}
		The SDC method, as initially defined in \cite{dutt:2000}, is an iterative method that uses two levels of discretizations.  The main goal of the method is to approximate the solution to \eqref{eq:ODE} on the outer time grid,
		\begin{equation}
		\label{eq:outer_grid}
			0 < t_1 < t_2 < \ldots < t_n < \ldots < t_N=T.	
		\end{equation}
		 		We denote the intervals  $I_n = [t_n, t_{n+1}]$ and the length of each interval by $\D t_n = t_{n+1} - t_n$. We subdivide each interval into subintervals, defined by $M+1$ subnodes denoted by $\tm$,
		\[ t_n = \tin{0} < \tin{1} < \ldots < \tm < \ldots < \tin{M} = t_{n+1}.\]
				These nodes are chosen to be the Gauss-Lobatto nodes over $I_n$ to maximize the accuracy of the interpolation below.  The length the subintervals is denoted by $\D \tm = \tin{m+1} - \tm$ and we denote each subinterval by $\Im = [\tm, \tin{m+1}]$.  
		
		To describe the SDC method, we narrow our focus to just one interval $[t_n, t_{n+1}]$.  We define the iterative procedure on this interval, which is then repeated for each interval over the domain.  We let $Y_{n,m}^k \approx y(\tm)$ denote the approximation to the exact solution at $\tm$ after $k$ iterations of the SDC method.  
		The spectral deferred correction method is implemented as follows.  We begin with an initial approximation $Y^0_{n,m}$ for each subnode $\tm$.  For our purposes, we let the initial approximation be the initial condition repeated across all subnodes. 
		We then proceed iteratively using the equation using either an explicit or an implicit method. The explicit SDC method is \cite{minion:2003},
		\begin{align}
			Y_{n,m+1}^{k+1} &= Y_{n,m}^{k+1} + \D \tm [f(Y_{n,m}^{k+1}, \tm) - f(Y_{n,m}^k,\tm)] + \intm \Sm f^k \,dt 
			\label{eq:SDC_explicit}	
		\end{align}		
		where $\Sm f^k$ is the polynomial interpolant  over the interval $I_n$ using the $M+1$ values $f(Y^k_{n,m},\tm)$. 
		\color{black}
		The implicit SDC method, suitable for stiff problems, is \cite{Emmett2012},
		\begin{align}
			Y_{n,m+1}^{k+1} &= Y_{n,m}^{k+1} + \D \tm [f(Y_{n,m+1}^{k+1}, t_{n,m+1}) - f(Y_{n,m+1}^k,t_{n,m+1})] + \intm \Sm f^k \,dt 
			\label{eq:SDC_implicit}	
		\end{align}				
		
		Both SDC methods converge with order  $O(\D t_n^{ \min\{K,M\}})$ where $K$ is the number of SDC iterations~\cite{dutt:2000}.
		 The stability properties of the SDC methods are  examined in \cite{dutt:2000,minion:2003}.
		 \color{black}

	\subsection{Finite Element Method}
		To implement adjoint-based error estimation, a global definition of the approximation in a variational setting is needed.  This may be achieved by using a finite element method.  Here we define the finite element method on the subinterval $\Im$.  
		We use the same grid as for the SDC method above in \eqref{eq:outer_grid} to define the set of piecewise continuous functions,
				\[ \Ccal^q = \{ w \in C^0([0,T]; \Rbb^d) : w|_{\Im} \in \Pcal^q(\Im),  0 \leq n \leq N-1 , 0 \leq m \leq M-1 \}, \]
		where $\Pcal^q(\Im)$ is the standard space of polynomials of degree less than or equal to $q$.  The continuous Galerkin method of order $q+1$, cG($q$), to solve \eqref{eq:ODE} is defined interval-wise by:
		
		Find $Y(t) \in \Ccal^q$ such that $Y(0) = y_0$ and for $n = 0, \ldots, N-1$ and $m = 0, \ldots, M-1$,
		\begin{align}
			\< \dot{Y}, v \>_{\Im} &= \<f(Y,t), v \>_{\Im}, \qquad \forall v \in \Pcal^{q-1}(\Im).
			 \label{eq:CG}
		\end{align}

	\subsection{Nodally Equivalent FEM Method}
		In this section, we develop a finite element method that is nodally equivalent to the SDC method approximation.  Nodal equivalence is an idea that was developed in \cite{collins:2015} for Runge-Kutta and multi-step methods and in  \cite{CEG+2015,chaudhry:2017} for IMEX schemes.  Two approximations are defined to be nodally equivalent if their approximations are equal at the nodes $\{t_n\}$ for $0 \leq n \leq N$.  
		\color{black}
		Since the FEM method is nodally equivalent to the SDC method, it inherits the stability and convergence properties of the SDC method at the nodes.
		\color{black}
		Methods are nodally equivalent if they generate nodally equivalent approximations.  Therefore, the finite element method we develop in this section will be a function $Y^k(t) \in \Ccal^q$ with the property, 
		\[ Y^k(\tm) = Y^k_{n,m}.\]
		
		As with the SDC method, the finite element will be defined iteratively.  We assume the finite element solution at the $k^{\text{th}}$ iteration, $Y^k(t)$, has been computed.  We then define the equivalent finite element method as:
		
		Find $Y^{k+1}(t) \in \Ccal^q$ such that $Y^{k+1}(0) = y_0$, and for $n = 0, \ldots, N-1$, and $m = 0, \ldots, M-1$,
										\color{black}
		\begin{align}
			\< \dot{Y}^{k+1}, v_m \>_{\Im} &= \< f(Y^{k+1},t) - f(Y^k,t),v_m \>_{R,\Im} \label{eq:equiv:FEM}  + \< \Sm f^k, v_m \>_{\Im},  
		\end{align}
		for all $v_m \in \Pcal^{q-1}(\Im)$. Here $\<\cdot, \cdot \>_{R,\Im}$ represents the  approximation of the integral using a quadrature rule. The two choices for quadrature we consider the left-hand rule and the right-hand rule, denoted $R = LHR$ and $R = RHR$ respectively.

		\color{black}
		We now show that the finite element method defined by \eqref{eq:equiv:FEM} is indeed nodally equivalent to the SDC method.
		
		\begin{thm} \label{thm:equiv}
			Assume $\{Y^{k+1}_{n,m}\}_{n,m=0}^{N,M}$ is a solution of the SDC  method defined by 
			\color{black} 
			either \eqref{eq:SDC_explicit} or \eqref{eq:SDC_implicit} 
			\color{black}  and $Y^{k+1}(t)$ is the  finite element solution defined by \eqref{eq:equiv:FEM} for 
			\color{black} 
			$R = LHR$ or $R = RHR$ respectively.  
			\color{black}
			Then $Y^{k+1}(\tm) = Y^{k+1}_{n,m}$ for all $n = 0 \ldots, N$ and $m = 0, \ldots, M$.	That is, the finite element method \eqref{eq:equiv:FEM} is nodally equivalent the the SDC method.
		\end{thm}

		\begin{proof}
		\color{black}
		We show the proof for the equivalence to \eqref{eq:SDC_explicit} when $R = LHR$ and note that the proof is identical for the implicit method, except for the use of $R = RHR$ in one particular step.
		\color{black} 
			We examine an arbitrary subinterval $\Im$ and set the test function $v_m = 1$ in equation \eqref{eq:equiv:FEM}.  First, we note that since $Y^{k+1}(t) \in \Pcal^q(\Im)$, we have that 
			\[ \< \dot{Y}^{k+1}(t), v_m \>_{\Im} = [Y^{k+1}(\tin{m+1}) -Y^{k+1}(\tm)].\]
			
			Next, we consider the quadrature term in \eqref{eq:equiv:FEM}.  Evaluating this with the chosen test function we obtain,
			\begin{equation}
			\label{eq:usage_of_lhr}
			\begin{aligned}			
				\< f(Y^{k+1},t),v_m \>_{LHR,\Im} &= \D \tm f(Y^{k+1}(\tm),\tm) \\
				\< f(Y^k,t),v_m \>_{LHR,\Im} &= \D \tm f(Y^{k}(\tm),\tm)
			\end{aligned}
			\end{equation}
			
			Finally, considering the last term, we see that,
			\begin{align}
				\< \Sm f^k, v_m \>_{\Im} = \intm \Sm f^k \,dt.
			\end{align}
			
			Combining all of this we see that when $v_m =1$, \eqref{eq:equiv:FEM} is written as,
			\begin{align}
				Y^{k+1}(\tin{m+1}) -Y^{k+1}(\tm) &= \D \tm [f(Y^{k+1}(\tm),\tm) - f(Y^{k}(\tm),\tm)] \\
				&\quad + \intm \Sm f^k \,dt. \notag
			\end{align}
			This is equivalent to the update formula \eqref{eq:SDC_explicit}.  Therefore the method defined by \eqref{eq:equiv:FEM} is nodally equivalent to the explicit SDC method.

			\color{black}
				As stated above, the proof for the implicit method \eqref{eq:SDC_implicit} is identical except for the use of $R=RHR$ in \eqref{eq:usage_of_lhr}.
			\color{black} 
		\end{proof}

		\begin{rem}
		Note that \eqref{eq:equiv:FEM} can be derived from the typical continuous Galerkin method \eqref{eq:CG}, with interesting use of quadrature.  To see this, consider the continous Galerkin method for $Y^{k+1}(t)$ from \eqref{eq:CG}:
		
		Find $Y^{k+1}(t) \in \Ccal^q$ such that for $m = 0, 1, \ldots, M$,
		\[ \<\dot{Y}^{k+1},v_m\>_{\Im}  = \<f(Y^{k+1},t),v_m\>_{\Im}  \hspace{.2in} \forall v_m \in \Pcal^{q-1}(\Im).\]
		Now add and subtract $\< \Sm f,v) \>_{\Im}$ to the right side,
		\begin{align}
			\<\dot{Y}^{k+1},v_m\>_{\Im} &= \<f(Y^{k+1},t),v_m\>_{\Im} \pm \< \Sm f^k,v_m\>_{\Im} \\
			&= \<f(Y^{k+1},t) - \Sm f^k,v_m\>_{\Im}+ \<\Sm f^k,v_m \>_{\Im}. \notag
		\end{align}
		Next, to obtain \eqref{eq:equiv:FEM}, we approximate the first integral on the right hand side with \color{black}
		either the left-hand rule or the right-hand rule.
		\color{black}
		Since the interpolation operator $\Sm$ does not change the value at the interpolation nodes, once quadrature is applied, the operator $\Sm$  in the first term disappears and we obtain the equivalent FEM method \eqref{eq:equiv:FEM},
		\color{black}
		 with $R = LHR$ or $R=RHR$ corresponding to the choice of the left-hand rule or the right-hand rule respectively.
		 \color{black}

	\end{rem}

	\subsection{Convergence of the Finite Element Method}
		In this  section we show that the finite element solution $Y^k(t)$ converges to the exact solution $y(t)$ as $\D t_n \goesto 0$.  We first note that the SDC method is at least second order accurate at the nodes $\tm$ as long as at least two iterations are done.  Since $Y^k(t)$ agrees with the SDC approximation at these nodes, we have that,
		\begin{align}
			| Y^k(\tm) - y(\tm) | \leq C \D t_n^2	
		\end{align}
		for $m = 0, \ldots, M$ for each interval $I_n$ in the outer discretization. Here $| \cdot |$ denotes the standard Euclidean norm on $\mathbb{R}^d$.
		We next show convergence of the finite element solution in the infinity (or max) norm, that is, for all values of $t$ and not just at the nodes $\tm$.

				\begin{thm}
			If $Y^k(t)$ is a cG(1) solution to \eqref{eq:equiv:FEM}, then for sufficiently small $\D t_n$ and $k \geq 2$,
			\begin{align}
				\|y - Y^k \|_{\infty} = \sup_{t \in [0,T]}|y(t) - Y^k(t) | \leq (C_1 + C)\D t_n^2	,
			\end{align}
			where $C_1$ and $C$ are constants independent of $\D t_n$.
		\end{thm}
		\begin{proof}
			The proof is identical to the proof of Theorem 2 found in \cite{chaudhry:2017}.
		\end{proof}
		
		Therefore we see that as $\D t_n \goesto 0$, $Y^k(t) \goesto y(t)$.  	
	\subsection{Order of the Finite Element Method}
		While using a nodally equivalent cG(1) finite element approximation will certainly converge  and agree with the SDC approximation at the nodes $\tm$, the accuracy of $Y^k(t)$ over the whole domain $[0,T]$ will not necessarily equal the accuracy of the SDC approximation.  The accuracy at the nodes  $t_n$ will always be the same, meaning the finite element method could have super-convergence at the nodes.  However, if our quantity of interest depends on the entire domain and not only the nodes, then the order of accuracy of a cG(1) approximation could be quite lower than the SDC order of accuracy.

		To obtain an sufficiently accurate approximation to the quantity of interest, we tailor the order $q$ of the finite element approximation depending on the accuracy of the SDC approximation.  A cG($q$) finite  element approximation to \eqref{eq:equiv:FEM} has an $L^2$ error of $E_{FEM} =  O(\D t_{n,M}^{q+1})$, where $\D t_{n,M} = \max(\D \tm)$.  The aim is to choose a value of $q$ such that $Y^k(t)$ has an error $E_{SDC} =   O(\D t_n^{ \min\{K,M\}})$, where $E_{SDC}$ is  the accuracy of the SDC method at the nodes after $K$ iterations.  
		
		To obtain the appropriate value of $q$, we equate the errors of the  two  methods, ignoring any error constants.  In addition, we make the approximation that $\D t_{n,M} \approx \frac{\D t_n}{M}$.  Any inaccuracies due to these approximations will be  canceled by rounding up to the nearest integer.  Equating  $E_{FEM}$ and $E_{SDC}$ we have,
		\begin{align}
				\left(\frac{\D  t_n}{M}\right)^{q+1} = \D t^{\min\{K,M\}}_n.  \label{eq:equate:err}
		\end{align}
		We solve \eqref{eq:equate:err} for $q$, then round up to the nearest integer.  This leads to,
		\begin{align}
			q = \ceil[\Bigg]{\frac{\min\{K,M\} \ln(\D t_n)}{\ln(\D t_n) - \ln(M)} - 1}.  \label{eq:FEM:q}
		\end{align}

		It should be noted that using \eqref{eq:FEM:q} to determine the order of the finite element method does not guarrantee $q+1$ order convergence in an $L^2$ sense.  This calculation is only done to attempt to have the finite element solution agree with the SDC accuracy over the entire domain, as opposed to only at the nodes.

	\subsection{Implementing a cG($q$) Finite Element Approximation}
		Since we now know what the order of our finite element approximation must be to obtain consistent accuracy with the SDC method, we next discuss how to obtain a cG($q$) finite element approximation using \eqref{eq:equiv:FEM} for $q > 1$.  
		
		Assume that we have obtained a solution $Y^K_{n,m}$ using $K$ iterations of the SDC method.  Note that the SDC approximation gives all the information needed to obtain the solution to \eqref{eq:equiv:FEM} using a cG(1) scheme.  To calculate a solution $Y^K(t)$ to \eqref{eq:equiv:FEM} using a cG($q$) scheme where $q > 1$, then $q-1$ additional values are needed within each subinterval $\Im$ at nodes denoted by, 
		 
		\[ \tm = \t_0  < \t_1 < \ldots < \t_{q-1} < \t_q = \tin{m+1}.\]
		
		We know the values $Y^K(\t_0)$ and $Y^K(\t_q)$ through the SDC approximation.  To determine the intermediate nodes, $\t_1, \ldots, \t_{q-1}$, we use different test functions in \eqref{eq:equiv:FEM}.  In Theorem \ref{thm:equiv} we used the test function $v_m = 1$ to show equivalency.  We obtain  $q-1$  equations  to determine the intermediate values of $Y^K(t)$  using the Lagrange  basis functions of degree $q-1$, $\{\ell_i^{q-1}(t)\}_{i=0}^{q-2} \subset \Pcal^{q-1}(\Im)$ centered at $q$ equally  spaced nodes within $\Im$ for the test functions.  For this to work, the set $\{\ell_i^{q-1}(t)\}_{i=0}^{q-2}$ must be linearly independent from the test function $v(t) =1$  used to determine equivalency.  
		
		\begin{lem}
			Let $\{\ell_i^{q-1}\}_{i=0}^{q-1} \subset \Pcal^{q-1}(\Im)$ denote the Lagrange basis functions based on $q$ uniformly spaced nodes in the interval $\Im$ including  the boundary.  Then the set of polynomials $S = \{\ell_i^{q-1}\}_{i=0}^{q-2} \cup \{1\}$ are linearly independent.
		\end{lem}
		
		\begin{proof}
			For the sake of contradiction, assume the set $S$ is linearly dependent.  Then there exists $c_i$  not all zero such that the polynomial
			\[ p(t) = c_0 \ell_0^{q-1}(t) + \ldots +  c_{q-2} \ell_{q-2}^{q-1}(t)	+ c_{q-1} \equiv 0. \]
			Evaluating $p(t)$ at $\tin{m+1}$ we obtain,
			\[ p(\tin{m+1}) = c_{q-1} = 0. \]
			This implies,
			\[ p(t) = c_0 \ell_0^{q-1}(t) + \ldots +  c_{q-2} \ell_{q-2}^{q-1}(t) \equiv 0,\]
			which implies the Lagrange basis functions are linearly dependent, which is a contradiction.
		\end{proof}

		Let $\{\ell^q_i(t)\}_{i=0}^q$ be the degree $q$  Lagrange basis polynomials on $\Im$ centered at  the nodes $\{ \t_i \}_{i=0}^{q}$.  Then we can write our finite element approximation as,
		\begin{align}
			Y^K(t) = \sum_{i=0}^q Y^K_i \ell^q_i(t), \label{eq:FEM:Yk}
		\end{align}
		where $Y^K_i = Y^K(\t_i)$.  We  know $Y^K_0$ and $Y^K_q$ from the SDC approximation.  Inserting the test functions $\{\ell_i^{q-1}(t)\}_{i=0}^{q-2}$ into \eqref{eq:equiv:FEM} 
		\color{black}
		when $R = LHR$
		\color{black}
		, we obtain the system of equations,
		\begin{align}
			\int_{\tm}^{\tin{m+1}} \dot{Y}^{K}(t)\ell_i^{q-1}(t) \,dt &= \left( \int_{\tm}^{\tin{m+1}} (f(Y^{K}(t),t) - f(Y^{K-1}(t),t))	\ell_i^{q-1}(t) \,dt \right)_{LHR} \\
		&  + \int_{\tm}^{\tin{m+1}} (\Sm f^{K-1}) \ell_i^{q-1}(t) \,dt \qquad i = 0, \ldots, q-2 \notag 
		\end{align}
		Now using \eqref{eq:FEM:Yk}, we obtain,
		\begin{align}
			\sum_{j=0}^q Y_j^{K} \int_{\tm}^{\tin{m+1}} \dot{\ell}_j^q(t) \ell_i^{q-1}(t) \,dt &= \D \tm [f(Y^{K}_0,\tm) - f(Y^{K-1}_0,\tm)] \ell_i^{q-1}(\tm) \\
		 &   + \sum_{j=0}^{M-1} f(Y^{K-1}(t_j), t_j) \int_{\tm}^{\tin{m+1}} L_j^m(t) \ell_i^{q-1}(t) \,dt \notag \\
		 &  \qquad i = 0, \ldots, q-2, \notag
			\end{align}
		where the functions $L_j(t)$ are the Lagrange basis functions through the subnodes $\tm$.  Rearranging, we can write this as, 
		\begin{align}
			A \mbfY^{K} = \mbfb_E + B \mbfF^{K-1} - \mbfb_{SDC} \label{eq:middle:nodes}
		\end{align}
		where $\mbfY^{K} = [Y^{K}_1, \ldots, Y^{K}_{q-1}]^T$, $\mbfF^{K-1} = [f(Y^{K-1}(t_{n,0}), t_{n,0}), \ldots, f(Y^{K-1}(t_{n,M}), t_{n,M})]^T$ is $f$ evaluated at the $\Im$ subnodes and,
		\begin{align}
			A_{ij} &= \int_{\tm}^{\tin{m+1}} \dot{\ell}_j^q(t) \ell_{i-1}^{q-1}(t)\,dt \qquad i,j = 1,\ldots, q-1 \\
			B_{ij} &= \int_{\tm}^{\tin{m+1}} L_j^m(t) \ell_{i-1}^{q-1}(t)\,dt \qquad i = 1,\ldots,q-1 \;\; j = 0,\ldots,M\\
			\mbfb_E &= [\D \tm (f(Y^{K}_0,\tm) - f(Y^{K-1}_0,\tm)), 0, \ldots, 0]^T \\
			\mbfb_{SDC,i} &= Y_0^{K} \int_{\tm}^{\tin{m+1}} \dot{\ell}_0^q(t) \ell_i^{q-1}(t) \,dt + Y_q^{K} \int_{\tm}^{\tin{m+1}} \dot{\ell}_q^q(t) \ell_i^{q-1}(t) \,dt \\
			&  \qquad i = 1, \ldots, q-1.  \notag
		\end{align}
		  Solving \eqref{eq:middle:nodes} for the intermediate nodes gives the coefficients $Y_i^K$ in \eqref{eq:FEM:Yk} and hence the cG($q$) finite element solution $Y^K(t)$ to \eqref{eq:equiv:FEM}
		  \color{black}
		   when $R = LHR$.

		  For the implicit case corresponding to $R = RHR$, the only change is that the $\mbfb_E$ is zero. This is because the Langrange basis functions $\{\ell_i^{q-1}(t)\}_{i=0}^{q-2}$ are zero at the right end of $I_{n,m}$.
		  \color{black}

\section{\textit{A posteriori} Error Analysis} \label{sec:err}
	In this section, we derive an a posteriori error representation formula using variational analysis, adjoint operators and computable residuals.  In particular, we develop a formula for the error in  a QoI represented by \eqref{eq:qoi_definition}.
		For the rest of this section, $Y^K$ represents the final iterate of SDC equivalent finite element method \eqref{eq:equiv:FEM}, and $y$ represents the exact solution to \eqref{eq:ODE}. Moreover,  $e = y-Y^K$ represents the error in $Y^K$.  
	
	\subsection{Adjoint Problem}
	We define the adjoint problem as follows:
	\begin{align}
		-\dot{\f} &= \Hop^T \f + \psi, \hspace{.2in} T> t \geq 0 \label{eq:adj:ode}\\
		\f(T) &= \psi_T, \notag	
	\end{align}
	where 
	\begin{align}
	\label{eq:linearization}
		\Hop = \int_0^1 \pdd{f(z)}{y} \, ds,	
	\end{align}
	for $z = sy + (1-s)Y^K$.  This operator has the property, by the chain rule, that $\Hop(y-Y^K) = f(y) - f(Y^K)$.  Note that the adjoint problem is solved backwards from $T$ to $0$, with the initial condition given at the time $t=T$.

	\subsection{Error Representation Formula}
	We give two error representation formulas for $Q(y-Y^K)$.  The first is well known, but is given for completeness.
	\begin{thm} \label{thm:basic:err}
		  The error in the quantity of interest, $Q(e)$, is given by,
		\begin{align}
			\<e,\psi \>_{[0,T]} + (e(T), \psi_T) &= \< f(Y^K,t) - \dot{Y}^K, \f \>_{[0,T]}.	
		\end{align}

	\end{thm}

	\begin{proof}
		We begin by substituting the adjoint ODE \eqref{eq:adj:ode} for $\psi$ in the error of the quantity of interest.
		
		\begin{align}
			\< e , \psi \>_{[0,T]} &= \<e, -\dot{\f} - \Hop^T \f \>_{[0,T]} \\
			&= 	\<\dot{e} - \Hop e, \f \>_{[0,T]} - (e(T), \psi_T) \notag \\
			&= \< \dot{y} - \dot{Y}^K - f(y,t) + f(Y^K,t), \f \>_{[0,T]} - (e(T), \psi_T) \notag \\
			&= \< f(Y^K,t) - \dot{Y}^K, \f \>_{[0,T]} - (e(T), \psi_T). \notag
		\end{align}
		Rearranging, we obtain the above error representation formula.

	\end{proof}

	Next, we consider a different error representation formula, which aims to separate the error into the various components attributed to different parts of the SDC method.

\color{black}

	\begin{thm} 
	\label{thm:err_contributions}

	\begin{align}
			\<e,\psi \>_{[0,T]} + (e(T), \psi_T) &=  E_{D} + E_M + E_K	. \label{eq:component:err_alt}
		\end{align}
		The error components are given by,
		\begin{align}
			E_{D} &= \sum_{n=0}^{N-1} \bigg\{ \sum_{m=0}^{M-1} \Big[  \< \Sm f^{K-1} - \dot{Y}^K, \f  - \pinm \f \>_{\Im} \\
			& \hspace{.75in}+  \<f(Y^{K},t) - f(Y^{K-1},t),   \f - \pinm \f\>_{R,\Im}\Big] \bigg\} \\
			E_M &= \sum_{n=0}^{N-1} \bigg\{ \sum_{m=0}^{M-1} \Big[ \<f(Y^K,t) - \Sm f^K, \f \>_{\Im} \Big] \bigg\}  \\
			E_K &=  \sum_{n=0}^{N-1} \bigg\{ \sum_{m=0}^{M-1} \Big[ \<f(Y^{K-1},t) - f(Y^K,t), \f\>_{R,\Im}  \\
			  & \hspace{.75in}+  \< \Sm f^{K} - \Sm f^{K-1}, \f \>_{\Im} \Big] \bigg\} \notag
		\end{align}

where $\pinm: C^1(\Im) \rightarrow \Pcal^{q-1}(\Im)$ is a nodal projection operator. 

	\end{thm}

	\begin{proof}
		Let $E = \<e,\psi \>_{[0,T]} + (e(T), \psi_T)$.
	 	Starting with the results from Theorem \ref{thm:basic:err}, we have that 
		\begin{align}
			E &= \sum_{n=0}^{N-1} \sum_{m=0}^{M-1}\< f(Y^K,t) - \dot{Y}^K, \f  \>_{\Im}.
		\end{align}
		By adding and subtracting  
		$\< \Sm f^{K-1} , \f  \>_{\Im} $, $ \< \Sm f^K, \f \>_{\Im}$, and $\<f(Y^{K-1},t) - f(Y^K,t), \f\> $	
		inside the double sum and re-arranging we arrive at
		\begin{align}
		E &= \sum_{n=0}^{N-1} \bigg\{ \sum_{m=0}^{M-1} \Big[  \< \Sm f^{K-1} - \dot{Y}^K, \f   \>_{\Im} \\
		& \hspace{.75in}+ \<f(Y^{K},t) - f(Y^{K-1},t),   \f \>_{R,\Im}\Big] \bigg\}  + E_M + E_K.  \label{eq:partial_err_rep}
		\end{align}

		Replacing $v_m$ by $\pinm \f \in \Pcal^{q-1}(\Im)$ in \eqref{eq:equiv:FEM},
		\begin{equation}
		\label{eq:gal_orthog}
		\begin{aligned}
		&\< \dot{Y}^{K}, \pinm \f \>_{\Im} - \< f(Y^{K},t) - f(Y^{K-1},t),\pinm \f \>_{R,\Im} \\
		& \hspace{.75in}- \< \Sm f^{K-1}, \pinm \f \>_{\Im}  = 0.
		\end{aligned}
		\end{equation}
		We combine \eqref{eq:partial_err_rep} and \eqref{eq:gal_orthog} which completes the proof.

	\end{proof}

\color{black}

By examining the three terms of the error representation formula \eqref{eq:component:err_alt}, we see that the terms correspond to the three parameters of the SDC method.  The first term $E_{D}$ represents the discretization error due to the SDC method. Note this term is simply the residual of the equivalent FEM method \eqref{eq:equiv:FEM} weighted by the adjoint solution.  
The term $E_M$ compares the approximation to an interpolation of the approximation through the subnodes $\tm$.  As the number of subnodes $M$ increases, the interpolant should converge to the true function, and this term should decrease.  Finally, the term $E_K$ considers the difference between two succesive iterations.  As the iterations converge, this term should decrease.  These terms can be used to fine tune all three parameters $\D t$, $M$, and $K$, depending on which term in the error representation formula is largest.

\subsection{Error Estimate and its accuracy}
\label{sec:adj_effec_rat}
The adjoint solution in Theorem \ref{thm:err_contributions} also needs to be approximated numerically. We  use the SDC finite element equivalent method to solve the adjoint equation~\eqref{eq:adj:ode}, where the order of the method is determined by equation \eqref{eq:FEM:q}.  Often it is desired that the adjoint equation is solved to a higher accuracy than the forward equation.  This may be achieved  in multiple ways when using the SDC method, e.g. by using extra iterations of the method (larger $K$), refining the discretization  (smaller $\Delta t_n$)  or a larger number of subintervals  (larger $M$) relative to the parameters used to obtain the SDC solution given by \eqref{eq:SDC_explicit} 
\color{black}
or \eqref{eq:SDC_implicit}.  
\color{black}
In this article we 
solve the adjoint problem on a temporal grid  with twice as many intervals than those used for the solution of the forward problem.

Additionally, the linearization \eqref{eq:linearization} requires the unknown true solution. In practice, we create an approximate adjoint problem by linearizing around a computed numerical solution and then compute the adjoint solution numerically. Hence, we obtain an approximate error estimate from the exact  error representation. This is well established in the literature~\cite{estep2, Fish:2009,Estep:Larson:00} and produces robust accuracy. This approach is employed for all of the adjoint problems discussed in this paper. 

The numerical approximation of the adjoint solution and the effect of linearization leads to an error estimate from the error representation~\eqref{eq:component:err_alt}.
For the sake of simplicity, we do not write approximate versions of the exact error representations.
To measure how well the error formula estimates the true error, we use the effectivity ratio
\[ \Ecal = \frac{\text{Exact Error}}{\text{Error Estimate}}. \]
An effectivity ratio close to one indicates that the error estimate is accurate.

\subsection{Algorithm for Error Estimation}
The full algorithm for \emph{a posteriori} error estimation for the SDC method is provided in Algorithm \ref{alg:error_estimation}.
\begin{algorithm}[H]
\caption{Adjoint-based \emph{a posteriori} error estimation for SDC method}
\label{alg:error_estimation}
\begin{algorithmic}
    \State Solve primal problem at subnodes $\tm$   \hfill (see \eqref{eq:SDC_explicit} or \eqref{eq:SDC_implicit})
    \State For $q >1$, compute FEM solution \hfill (see \eqref{eq:middle:nodes})
    \State Approximate solution of adjoint problem     \hfill (see \eqref{eq:adj:ode})
   	\State Estimate error and its contributions  \hfill (see Theorem~\ref{thm:err_contributions})
\end{algorithmic}
\end{algorithm}

\section{Numerical Experiments} \label{sec:results}
	We now demonstrate the accuracy of the error estimate derived in Theorem~\ref{thm:err_contributions} numerically for a variety of linear and nonlinear problems as various parameters of the SDC method are changed. The first experiment, in \S \ref{sec:harmonoc_osc}, is on a linear second order ODE modeling a harmonic oscillator. The second experiment, in \S \ref{sec:vinograd},  is the Vinograd ODE which is an example of an non-autonomous system of ODEs. The third experiment, in \S \ref{sec:two_body}, concerns the nonlinear two-body problem.
	 We also illustrate the importance of using the appropriate order for the equivalent FEM in this example.
	 \color{black}
	 These first three problems are solved using the explicit SDC method in \eqref{eq:SDC_explicit}.
	 \color{black}
	Finally, we examine  the accuracy of the error estimate and its different contributions for the discretization of a parabolic partial differential equation in \S \ref{sec:heat_pde}. 
	\color{black}
	This is example of a stiff problem and is solved using the implicit SDC method in \eqref{eq:SDC_implicit}.
	\color{black}

	The numerical examples also illustrate the utility of the estimate \eqref{eq:component:err_alt} to distinguish the various contributions of the error. To keep the notation clear, we label each of these parameters for the solution of the forward problem (\eqref{eq:SDC_explicit} or \eqref{eq:SDC_implicit}) here.  The number of intervals is denoted by $N$, the (uniform) time-step $T/N$ is denoted by $\Delta t$, the  number of subintervals within each interval is denoted by $M$, while the total number of iterations $K$. 
	We expect all contributions of error to decrease as $\Delta t$ is decreased as a smaller $\Delta t$ should not only decrease the discretization contribution $E_D$ but also $E_M$ (since the interpolation is now over a smaller interval) and $E_K$ (since the iterations need to converge over a smaller interval). 
	Increasing $M$ should effect  $E_M$ but should not have a significant effect on the other two contributions. Similarly, by increasing $K$ we expect $E_K$  to decrease while there should be a less pronounced effect on $E_M$ and $E_D$.
	\color{black}
	The \textit{a posteriori} error estimate and error contributions may also be used for error control as we demonstrate in \S \ref{sec:err_cont_har_osc} for the case of the harmonic oscillator. We abbreviate the Error Estimate (see \S \ref{sec:adj_effec_rat}) as Err. Est.  in the tables indicating the numerical results throughout this section.
	\color{black}

\subsection{Linear Harmonic Oscillator}
\label{sec:harmonoc_osc}
		 The first ODE we consider is the linear harmonic oscillator,
		 \begin{align}
		 	m \ddh{x}{t}{2} + c \dd{x}{t} + k x = F_0 \cos(\omega t + \f_d),	
		 \end{align}
		 where $m$ is the mass of the object, $c$ and $k$ are the damping and stiffness coefficients respectively, $F_0 \cos(\omega t + \f_d)$ is an oscillatory forcing term, and $x(t)$  represents the position of the oscillating mass. These parameters are set as $m = 0.5$, $c = 1.0$, $k = 1.0$, $F_0 = 10$, $\omega = 20$ and $\f_d = 0$. The initial conditions are $x(0) = 0$ and $x'(0) = 1$.  
		 \color{black}
		 This numerical example was also explored in \cite{CET+2016} with regards to the error in the Parareal algorithm \cite{Maday2002387}.
		 \color{black}
		 This second order ODE is converted to a first order system  $\dot{y}+Ay=g(t)$ in the usual manner by introducing variables $y_1 = x$ and $y_2 = dx/dt$.
Rewriting as a system of first-order ODEs,  $\dot{y}+Ay=h(t)$, gives
\begin{equation*}
\begin{pmatrix}
\dot{y_1}(t) \\  \dot{y_2}(t)
\end{pmatrix}
+
\begin{pmatrix}
0    & -1 \\  k/m  & c/m
\end{pmatrix}
\begin{pmatrix}
y_1(t) \\ y_2(t)
\end{pmatrix}
=
\begin{pmatrix}
0 \\  F_0 \cos(\omega t + \f_d)
\end{pmatrix}.
\end{equation*}
The initial conditions for the first-order system  are $[y_1(0), y_2(0)]^T = [0, 1]^T$. The QoI is specified by setting $\psi = [1 \quad1]^T$ and $\psi_T = [1, 0]^T$. The final time for the simulation is chosen as $T=5$.

The results for the estimated error, effectivity ratio and the various contributions to the error are shown in Tables~ \ref{tab:problem_Linear Harmonic Oscillator_R_vary_J_2_M_3}, \ref{tab:problem_Linear Harmonic Oscillator_R_40_J_vary_M_4} and \ref{tab:problem_Linear Harmonic Oscillator_R_40_J_4_M_vary}. In all our results, we see that the error estimate is accurate, indicated by the value of the effectivity $\Ecal$ being close to $1.0$. 

Table~\ref{tab:problem_Linear Harmonic Oscillator_R_vary_J_2_M_3} examines the error behavior as $\Delta t$ is decreased. As $\Delta t$ decreases, so does the error and its different components. The quantity $E_D$, which captures the error due to SDC discretization, is expected. The quantity $E_M$ decreases as now each sub-interval has smaller length. The quantity $E_K$ decreases because now the interval size is smaller and iterations converge faster.
\color{black}
Since $K=2$ we expect the component $E_D$ to decrease as $O(\D t^{2})$. This is indeed the case as seen from the values in Table~\ref{tab:problem_Linear Harmonic Oscillator_R_vary_J_2_M_3}. The cancellation of error among different terms means that the total decrease in the error is much greater than $O(\D t^{2})$.
\color{black}

Table \ref{tab:problem_Linear Harmonic Oscillator_R_40_J_vary_M_4} examines the error as the number of iterations $K$ is increased. The error initially decreases and then stagnates. This stagnation is explained by observing the different error contributions. Increasing $K$ initially decreases all components of error. However, around
the fourth iteration, $E_M$ 
and $E_D$
stop decreasing and become the dominant contributions to the error. Iterating further decreases $E_K$ but not $E_M$ 
or $E_D$
and hence the total error does not decrease either.

Table \ref{tab:problem_Linear Harmonic Oscillator_R_40_J_4_M_vary} examines what happens when the number of subintervals $M$ is varied. Increasing $M$ decreases  $E_M$ as expected. Increasing $M$ also decreases $E_D$ initially, but this contribution stops decreasing as $M$ increases. Moreover, increasing $M$ does not have a significant effect on $E_K$, which becomes the dominant contribution to the error as expected.

\begin{table}[!ht]
\centering
\begin{tabular}{c|c|c|c|c|c}
\toprule
$dT$ & Err. Est.  & $\Ecal$ &  $E_{D}$ &  $E_M$ & $E_K$   \\
\midrule
0.5000&	 1.23E-01&		  1.00&	 -2.00E-01&	 2.42E-01&	 8.15E-02 \\
0.2500&	 2.22E-03&		  0.99&	 -3.15E-02&	 3.06E-02&	 3.11E-03 \\
0.1250&	 -1.10E-04&	  1.01&	 -7.14E-03&	 6.94E-03&	 8.96E-05 \\
0.0625&	 -8.73E-05&	  1.00&	 -1.74E-03&	 1.70E-03&	 -4.56E-05 \\
\bottomrule
\end{tabular}
\caption{Results for the Linear Harmonic Oscillator problem in \S~\ref{sec:harmonoc_osc} as $\Delta t$ is varied. Here $M = 3, K = 2$.}
\label{tab:problem_Linear Harmonic Oscillator_R_vary_J_2_M_3}
\end{table}

\begin{table}[!ht]
\centering
\begin{tabular}{c|c|c|c|c|c}
\toprule
$K$ & Err. Est.  & $\Ecal$ &  $E_{D}$ &  $E_M$ & $E_K$   \\
\midrule
2&	 -1.42E-04&	  1.00&	 -4.36E-03&	 4.24E-03&	 -1.49E-05\\
3&	 -9.75E-05&	  1.00&	 -4.32E-03&	 4.24E-03&	 -1.14E-05\\
4&	 3.85E-07&		  1.00&	 -2.05E-07&	 2.05E-07&	 3.86E-07\\
5&	 -1.47E-08&	  1.00&	 -2.20E-07&	 2.09E-07&	 -4.14E-09\\
6&	 -1.01E-08&	  1.00&	 -2.19E-07&	 2.10E-07&	 -1.62E-10\\
7&	 -9.90E-09&	  1.00&	 -2.19E-07&	 2.10E-07&	 8.29E-12\\
8&	 -9.91E-09&	  1.00&	 -2.19E-07&	 2.10E-07&	 -1.26E-13\\
\bottomrule
\end{tabular}
\caption{Results for the Linear Harmonic Oscillator problem in \S~\ref{sec:harmonoc_osc} as the number of iterations $K$ is varied. Here $M = 4, \Delta t = 0.12$.}
\label{tab:problem_Linear Harmonic Oscillator_R_40_J_vary_M_4}
\end{table}

\begin{table}[!ht]
\centering
\begin{tabular}{c|c|c|c|c|c}
\toprule
$M$ & Err. Est.  & $\Ecal$ &  $E_{D}$ &  $E_M$ & $E_K$   \\
\midrule
2&	 2.96E-02&	  1.01&	 -1.03E-02&	 1.24E-02&	 2.75E-02\\
3&	 2.43E-02&	  1.01&	 -4.60E-03&	 7.15E-03&	 2.17E-02\\
4&	 2.02E-02&	  1.00&	 -2.23E-03&	 4.33E-03&	 1.81E-02\\
5&	 1.72E-02&	  1.00&	 -1.12E-03&	 2.90E-03&	 1.55E-02\\
6&	 1.50E-02&	  1.00&	 -5.37E-04&	 2.07E-03&	 1.35E-02\\
7&	 1.33E-02&	  1.00&	 -2.05E-04&	 1.56E-03&	 1.19E-02\\
8&	 1.19E-02&	  1.00&	 -6.91E-06&	 1.21E-03&	 1.07E-02\\
9&	 1.08E-02&	  1.00&	 1.16E-04&	 9.69E-04&	 9.68E-03\\
10&	 9.83E-03&	  1.00&	 1.95E-04&	 7.93E-04&	 8.84E-03\\
11&	 9.04E-03&	  1.00&	 2.45E-04&	 6.60E-04&	 8.14E-03\\
\bottomrule
\end{tabular}
\caption{Results for the Linear Harmonic Oscillator problem in \S~\ref{sec:harmonoc_osc} as the number of sub-intervals $M$ is varied. Here $K = 1, \Delta t = 0.12$.}
\label{tab:problem_Linear Harmonic Oscillator_R_40_J_4_M_vary}
\end{table}
 
\color{black}
\subsubsection{Error Control}
\label{sec:err_cont_har_osc}
The error estimate and the different contributions to the error may be used for error control. The aim of error control is to adapt numerical parameters to achieve a certain tolerance, TOL, for the error. If the Estimated Error is greater than TOL then we have the option of either decreasing $\Delta t$ or increasing $M$ or $K$ to lower the error. In this example, a simple metric  is chosen for choosing the next set of numerical parameters. If $E_{D}$ is the dominant component then $\Delta t$ is decreased by half, if $E_M$ is the dominant contribution then $M$ is increased by one, and if $E_K$ is the dominant contribution then $K$ is increased by one. The results of applying this strategy to the Linear Harmonic Oscillator problem are shown in Table~\ref{tab:problem_Linear Harmonic Oscillator_error_control_pnum_13} where we chose TOL$=1E-04$.

Initially, the error is dominated by $E_K$, and so $K$ is increased by one. Then the dominant contribution if $E_M$ for the next couple of steps and hence $M$ is increased. After this, $E_D$ is the dominant contribution and hence $\Delta t$ is reduced by a factor of two till the error reaches the desired tolerance of $1E-04$.

\begin{table}[!ht]
\centering
{\color{black} 
\begin{tabular}{c|c|c|c|c|c|c}
\toprule
 Err. Est. & $dT$ & $M$  & $K$ & $E_{D}$ &  $E_M$ & $E_K$       \\
\midrule
-2.04E+00 & 	 0.5000& 	 2 &	 1 &	  8.57E-02 &	 -1.00E+00 &	 -1.13E+00\\
-1.01E+00 & 	 0.5000& 	 2 &	 2 &	  -2.70E-01 &	 -9.28E-01 &	 1.92E-01\\
1.23E-01 & 	 0.5000& 	 3 &	 2 &	  -2.00E-01 &	 2.42E-01 &	 8.15E-02\\
3.10E-02 & 	 0.5000& 	 4 &	 2 &	  -1.26E-01 &	 1.10E-01 &	 4.64E-02\\
4.82E-04 & 	 0.2500& 	 4 &	 2 &	  -1.92E-02 &	 1.84E-02 &	 1.29E-03\\
-1.42E-04 & 	 0.1250& 	 4 &	 2 &	  -4.36E-03 &	 4.24E-03 &	 -1.49E-05\\
-6.20E-05 & 	 0.0625& 	 4 &	 2 &	  -1.06E-03 &	 1.04E-03 &	 -3.63E-05\\
\bottomrule
\end{tabular}
}
\caption{Results for the Linear Harmonic Oscillator problem for error control. Here TOL$=1E-04$. }
\label{tab:problem_Linear Harmonic Oscillator_error_control_pnum_13}
\end{table}

More sophisticated error control techniques based on utilizing cancellation of error and employing non-uniform grid sizes may also be devised, see \cite{Chaudhry17,JESVS15}.

\color{black}
\subsection{Vinograd ODE}
\label{sec:vinograd}
The Vinograd problem is a non-autonomous system  of two ODEs,
\begin{equation} \label{eq:vinograd}
\begin{aligned}
&\begin{bmatrix} \dot{y_1} \\ \dot{y_2}\end{bmatrix}
= -\begin{bmatrix} 1 + 9 \cos^2(6t) - 6 \sin(12t) & -12 \cos^2(6t) - \frac{9}{2} \sin(12t) \\ 12 \sin^2(6t) - \frac{9}{2} \sin(12t) & 1 + 9 \sin^2(6t) + 6 \sin(12t) \end{bmatrix} \begin{bmatrix}
  y_1 \\ y_2\end{bmatrix},
\end{aligned}
\end{equation}
The  initial conditions are chosen as $[y_1(0), y_2(0)]^\Tcal = [-1, 3]^T$. The analytical solution for this problem is given in~\cite{eehj1}. The QoI is specified by choosing $\psi_T = [1, 1]^T, \psi_{[0,T]} = [1, 1]^T$. This QoI gives a weighted average error of the solution components over $[0,T]$ and at the final time, which is chosen to be $T=2$.
\color{black}
 The Vinograd ODE was also used as a numerical example in \textit{a posteriori} analysis of iterative schemes in \cite{CEGT15b}.
 \color{black}

The results for the estimated error, effectivity ratio and the various contributions to the error are shown in Tables~ \ref{tab:problem_Vinograd_R_vary_J_2_M_3}, \ref{tab:problem_Vinograd_R_20_J_vary_M_3} and \ref{tab:problem_Vinograd_R_20_J_3_M_vary}. The results are qualitatively similar to the linear harmonic oscillator problem. Again, the effectivity ratios are close to $1.0$.

\begin{table}[!ht]
\centering
\begin{tabular}{c|c|c|c|c|c}
\toprule
$dT$ & Err. Est.  & $\Ecal$ &  $E_{D}$ &  $E_M$ & $E_K$   \\
\midrule
0.1000&	 9.30E+00&	  1.01&	 -1.95E+00&	 -4.43E+00&	 1.57E+01 \\
0.0500&	 4.07E+00&	  1.01&	 1.87E-01&	 -1.13E+00&	 5.01E+00 \\
0.0250&	 1.27E+00&	  1.00&	 1.57E-01&	 -2.87E-01&	 1.40E+00 \\
0.0125&	 3.51E-01&	  1.00&	 5.47E-02&	 -7.22E-02&	 3.69E-01 \\
\bottomrule
\end{tabular}
\caption{Results for the Vinograd problem in \S~\ref{sec:vinograd} as $\Delta t$ is varied. Here $M = 3, K = 2$.}
\label{tab:problem_Vinograd_R_vary_J_2_M_3}
\end{table}

\begin{table}[!ht]
\centering
\begin{tabular}{c|c|c|c|c|c}
\toprule
$K$ & Err. Est.  & $\Ecal$ &  $E_{D}$ &  $E_M$ & $E_K$   \\
\midrule
2&	 9.30E+00&	  1.01&	 -1.95E+00&	 -4.43E+00&	 1.57E+01\\
3&	 -1.99E+00&	  1.00&	 2.31E-01&	 7.92E-01&	 -3.02E+00\\
4&	 4.55E-01&	  1.00&	 9.44E-03&	 -1.63E-01&	 6.08E-01\\
5&	 -1.12E-01&	  1.00&	 -7.06E-03&	 2.97E-02&	 -1.35E-01\\
6&	 2.81E-02&	  1.00&	 8.64E-03&	 -1.20E-02&	 3.14E-02\\
7&	 -7.23E-03&	  1.00&	 2.83E-03&	 -2.53E-03&	 -7.53E-03\\
8&	 1.69E-03&	  1.00&	 4.59E-03&	 -4.74E-03&	 1.84E-03\\
9&	 -5.66E-04&	  1.00&	 4.10E-03&	 -4.21E-03&	 -4.54E-04\\
10&	 2.85E-06&	  1.08&	 4.23E-03&	 -4.34E-03&	 1.13E-04\\
\bottomrule
\end{tabular}
\caption{Results for the Vinograd problem in \S~\ref{sec:vinograd}  as the number of iterations $K$ is varied. Here $M = 3, \Delta t = 0.10$.}
\label{tab:problem_Vinograd_R_20_J_vary_M_3}
\end{table}

\begin{table}[!ht]
\centering
\begin{tabular}{c|c|c|c|c|c}
\toprule
$M$ & Err. Est.  & $\Ecal$ &  $E_{D}$ &  $E_M$ & $E_K$   \\
\midrule
2&	 7.82E+00&	  1.01&	 3.38E-01&	 -2.14E+00&	 9.63E+00\\
3&	 4.07E+00&	  1.01&	 1.87E-01&	 -1.13E+00&	 5.01E+00\\
4&	 2.35E+00&	  1.00&	 9.37E-02&	 -6.98E-01&	 2.96E+00\\
5&	 1.52E+00&	  1.00&	 5.18E-02&	 -4.72E-01&	 1.94E+00\\
6&	 1.06E+00&	  1.00&	 3.24E-02&	 -3.40E-01&	 1.37E+00\\
7&	 7.81E-01&	  1.00&	 2.16E-02&	 -2.56E-01&	 1.02E+00\\
8&	 5.98E-01&	  1.00&	 1.49E-02&	 -2.00E-01&	 7.83E-01\\
9&	 4.73E-01&	  1.00&	 1.07E-02&	 -1.60E-01&	 6.22E-01\\
\bottomrule
\end{tabular}
\caption{Results for the Vinograd problem in \S~\ref{sec:vinograd}  as the number of sub-intervals $M$ is varied.  Here $K = 2, \Delta t = 0.05$.}
\label{tab:problem_Vinograd_R_20_J_3_M_vary}
\end{table}

	\subsection{Two Body Problem}
	\label{sec:two_body}
		The next example we consider is a nonlinear system of ODEs; the well known two body problem,
		\begin{align}
			\begin{cases}
				y_1' &= y_3 \\
				y_2' &= y_4 \\
				y_3' &= \frac{-y_1}{(y_1^2 + y_2^2)^{3/2})} \\
				y_4' &= \frac{-y_2}{(y_1^2 + y_2^2)^{3/2})}
			\end{cases}
		\end{align}
		with the initial condition of $y(0) = [0.4, 0, 0, 2.0]^T$.  This is a complex nonlinear system that models the motion of two point bodies under the influence of gravitational force from the other. 
		 With these initial conditions, the exact solution can be given analytically as 
		\begin{align}
			y = \left[\cos(\tau) - 0.6, 0.8\sin(\tau), \frac{-\sin(\tau)}{1-0.6\cos(\tau)}, \frac{0.8\cos(\tau)}{1-0.6\cos(\tau)} \right]^T	
		\end{align}
		where $\tau$ depends on the independent variable $t$ and solves the equation \mbox{$\tau - 0.6\sin(\tau) = t$.} 
		\color{black}
		This problem was studied in the context of explicit time-stepping methods in \cite{collins:2015}.
		\color{black}
		
				The final time is set to  $T=2$.  
		The quantity of interest is chosen to be $\psi(t) = [1,1,0,0]^T$ and $\psi_T = [1,1,0,0]^T$. The results for the estimated error, effectivity ratio and the various contributions to the error are shown in Tables~ \ref{tab:problem_Two Body_R_vary_J_2_M_3}, \ref{tab:problem_Two Body_R_20_J_vary_M_3} and \ref{tab:problem_Two Body_R_20_J_3_M_vary}. The results are qualitatively similar to the previous two linear problems.

\begin{table}[!ht]
\centering
\begin{tabular}{c|c|c|c|c|c}
\toprule
$dT$ & Err. Est.  & $\Ecal$ &  $E_{D}$ &  $E_M$ & $E_K$   \\
\midrule
0.2000&	 -2.88E-01&	  0.99&	 5.95E-02&	 -9.65E-02&	 -2.51E-01 \\
0.1000&	 -7.83E-02&	  1.00&	 2.09E-02&	 -2.41E-02&	 -7.51E-02 \\
0.0500&	 -1.92E-02&	  1.00&	 5.79E-03&	 -6.03E-03&	 -1.90E-02 \\
0.0250&	 -4.69E-03&	  1.00&	 1.50E-03&	 -1.51E-03&	 -4.68E-03 \\
\bottomrule
\end{tabular}
\caption{Results for the Two Body problem in \S~\ref{sec:two_body} as $\Delta t$ is varied. Here $M = 3, K = 2$.}
\label{tab:problem_Two Body_R_vary_J_2_M_3}
\end{table}

\begin{table}[!ht]
\centering
\begin{tabular}{c|c|c|c|c|c}
\toprule
$K$ & Err. Est.  & $\Ecal$ &  $E_{D}$ &  $E_M$ & $E_K$   \\
\midrule
2&	 -5.32E-01&	  1.00&	 2.09E-02&	 -2.41E-02&	 -7.51E-02\\
3&	 -7.83E-02&	  1.00&	 -2.16E-04&	 -4.70E-04&	 -2.52E-03\\
4&	 -3.20E-03&	  1.00&	 2.89E-05&	 -8.45E-05&	 -4.35E-04\\
5&	 -4.91E-04&	  1.00&	 3.66E-05&	 -3.58E-05&	 -4.70E-05\\
6&	 -4.62E-05&	  0.96&	 3.68E-05&	 -3.01E-05&	 -5.95E-06\\
7&	 7.84E-07&	  0.99&	 3.68E-05&	 -2.98E-05&	 -3.26E-07\\
8&	 6.64E-06&	  1.00&	 3.68E-05&	 -2.98E-05&	 -4.68E-08\\
\bottomrule
\end{tabular}
\caption{Results for the Two Body problem in \S~\ref{sec:two_body} as the number of iterations $K$ is varied. Here $M = 3, \Delta t = 0.10$.}
\label{tab:problem_Two Body_R_20_J_vary_M_3}
\end{table}

\begin{table}[!ht]
\centering
\begin{tabular}{c|c|c|c|c|c}
\toprule
$M$ & Err. Est.  & $\Ecal$ &  $E_{D}$ &  $E_M$ & $E_K$   \\
\midrule
2&	 -9.25E-01&	  1.18&	 -1.32E-01&	 -4.48E-02&	 -5.02E-01\\
3&	 -6.80E-01&	  1.11&	 -1.27E-01&	 -2.36E-02&	 -3.81E-01\\
4&	 -5.32E-01&	  1.08&	 -1.15E-01&	 -1.46E-02&	 -3.09E-01\\
5&	 -4.38E-01&	  1.05&	 -1.02E-01&	 -9.85E-03&	 -2.59E-01\\
6&	 -3.71E-01&	  1.04&	 -9.18E-02&	 -7.09E-03&	 -2.23E-01\\
7&	 -3.21E-01&	  1.03&	 -8.28E-02&	 -5.34E-03&	 -1.95E-01\\
8&	 -2.83E-01&	  1.03&	 -7.54E-02&	 -4.17E-03&	 -1.73E-01\\
9&	 -2.53E-01&	  1.02&	 -6.90E-02&	 -3.34E-03&	 -1.56E-01\\
\bottomrule
\end{tabular}
\caption{Results for the Two Body problem in \S~\ref{sec:two_body} as the number of sub-intervals $M$ is varied. Here $K = 1, \Delta t = 0.10$.}
\label{tab:problem_Two Body_R_20_J_3_M_vary}
\end{table}

\subsubsection{Importance of choosing the correct FEM Order}

The two body example is also an ideal situation to illustrate the need to use a higher order continuous Galerkin method in the equivalent finite element method \eqref{eq:equiv:FEM}.  Here we solve the same problem but with SDC parameters of $M = 7$ and $K=8$ and the discretization fixed at $N = 64$, with a final time of $T = 8$. Finally, the QoI is changed to 
\[\psi = [e^{-(t-2)^2}, e^{-(t-2)^2}, 0, 0]^T, \qquad \psi_T = [1,1,0,0]^T.\]
 In Table \ref{tab:problem_Two Body_effect_q}, we consider solving the problem with a cG$(q)$ equivalent finite element method, but varying the value of $q$.  Note that for this case, \eqref{eq:FEM:q} gives $q=3$ for the correct value order of the finite element method.  We see that the effectivity ratio stabilizes at the calculated value, and gives poor results for the lower values. 
		
\begin{table}[!ht]
\centering
\begin{tabular}{c|c|c}
\toprule
$q$ & Exact Err.  & $\Ecal$    \\
\midrule
1&	 -8.83E-06&	  4.97E-08	 \\
2&	 -9.77E-09&	  1.25E-07		 \\
3&	 -9.079E-09&	  0.999	 \\
4&	 -9.079E-09&	  0.999	 \\
\bottomrule
\end{tabular}
\caption{Results for the Two Body problem in \S~\ref{sec:two_body} as the order of the finite element method $q$ is varied. Here $M = 7, K = 8,  \Delta t = 0.125$.}
\label{tab:problem_Two Body_effect_q}
\end{table}

\subsection{Heat Equation}
\label{sec:heat_pde}

We consider the one dimensional heat equation with boundary and initial conditions
\begin{equation*}\label{heat}
\begin{aligned}
&u_t(x,t) = u_{xx}(x,t)+\sin(\pi x)\cos(2\pi t), \quad (x,t) \in (0,1) \times (0,2], \\
&u(x,0) = 0, \quad x \in (0,1), \\
&u(0,t) = 0, \; u(1,t) = 0, \quad t \in (0,2].
\end{aligned}
\end{equation*}
This section analyzes the system of ordinary differential equations that arises from a spatial discretization of (\ref{heat}) using a central-difference method with a uniform partition. In particular using a uniform partition of the spatial interval $[0,1]$ with $d+2$ nodes:
\begin{equation*}
\{0=x_0<x_1<\dots<x_{d+2}=1 \}.
\end{equation*}
Since boundary values are specified, this semi-discretization leads to a system of $d$ first-order ODEs of the form 
\[
\dot{y}(t)=\frac{1}{h^2}Ay(t)+g(t),\] 
where $A$ is a tridiagonal matrix with $-2$'s on the diagonal and $1$'s on the sub and super diagonals, $h=\frac{1}{d+1}$ and $g(t)$ is vector of length $d$ with the $i$th component being $[g(t)]_i = \sin(\pi x_{i+1})\cos(2\pi t)$.
\color{black}
 In the numerical examples we choose $d=39$. This high-dimensional ODE system is stiff for values of $\Delta t$ sufficiently large, in the sense that the explicit SDC method, which is based on the forward Euler method, will be unstable. A large value of $\Delta t$ is chosen in all the numerical examples for this problem, to ensure the problem will be unstable using the explicit SDC method.  Therefore, we solve the problem with the implicit SDC method, and use the corresponding error representation formula.
\color{black}
The results for the estimated error, effectivity ratio and the various contributions to the error are shown in Tables \ref{tab:problem_Heat_R_vary_J_2_M_3}, \ref{tab:problem_Heat_R_10_J_vary_M_1} and \ref{tab:problem_Heat_R_10_J_2_M_vary}. For this problem, we do not have an analytical solution, so an accurate reference solution is computed using the \textit{odeint} function of the Scipy library~\cite{scipy}. The results for this high dimensional system again demonstrate the accuracy of the estimates with effectivity ratios being close to $1.0$. Moreover, the different contributions to the error indicate how the choice of parameters $\Delta t$, $M$ and $K$ effect the error.

\begin{table}[!ht]
\centering
{\color{black}
\begin{tabular}{c|c|c|c|c|c}
\toprule
$dT$ & Est. Err.  & $\Ecal$ &  $E_{D}$ &  $E_M$ & $E_K$   \\
\midrule
0.1000&	 1.49E-02&	  0.99&	 1.09E-02&	 2.79E-03&	 1.23E-03 \\
0.0500&	 5.47E-03&	  1.00&	 1.41E-03&	 7.92E-04&	 3.26E-03 \\
0.0250&	 1.69E-03&	  1.00&	 1.15E-04&	 2.06E-04&	 1.37E-03 \\
0.0125&	 4.72E-04&	  1.00&	 -8.56E-06&	 5.23E-05&	 4.28E-04 \\
\bottomrule
\end{tabular}
}
\caption{Results for the Heat problem as $\Delta t$ is varied. Here $M = 3, K = 2$.}
\label{tab:problem_Heat_R_vary_J_2_M_3}
\end{table}

\begin{table}[!ht]
\centering
{\color{black}
\begin{tabular}{c|c|c|c|c|c}
\toprule
$K$ & Est. Err.  & $\Ecal$ &  $E_{D}$ &  $E_M$ & $E_K$   \\
\midrule
2&	 1.47E-01&	  0.96&	 8.69E-02&	 5.98E-02&	 1.69E-04\\
3&	 6.16E-02&	  1.00&	 3.18E-03&	 5.90E-02&	 -6.07E-04\\
4&	 4.22E-02&	  1.00&	 -1.66E-02&	 5.89E-02&	 -1.75E-04\\
5&	 3.74E-02&	  1.01&	 -2.14E-02&	 5.89E-02&	 -4.44E-05\\
6&	 3.63E-02&	  1.01&	 -2.26E-02&	 5.89E-02&	 -1.11E-05\\
7&	 3.60E-02&	  1.01&	 -2.29E-02&	 5.89E-02&	 -2.75E-06\\
8&	 3.59E-02&	  1.01&	 -2.30E-02&	 5.89E-02&	 -6.82E-07\\
\bottomrule
\end{tabular}
}
\caption{Results for the Heat problem as the number of iterations $K$ is varied. Here $M = 1, \Delta t = 0.10$.}
\label{tab:problem_Heat_R_10_J_vary_M_1}
\end{table}

\begin{table}[!ht]
\centering
{\color{black}
\begin{tabular}{c|c|c|c|c|c}
\toprule
$M$ & Est. Err.  & $\Ecal$ &  $E_{D}$ &  $E_M$ & $E_K$   \\
\midrule
2&	 2.70E-02&	  0.98&	 2.15E-02&	 5.37E-03&	 7.17E-05\\
3&	 1.49E-02&	  0.99&	 1.09E-02&	 2.79E-03&	 1.23E-03\\
4&	 9.62E-03&	  0.99&	 6.67E-03&	 1.80E-03&	 1.15E-03\\
5&	 6.73E-03&	  1.00&	 4.45E-03&	 1.24E-03&	 1.04E-03\\
6&	 4.96E-03&	  1.00&	 3.21E-03&	 9.05E-04&	 8.53E-04\\
7&	 3.81E-03&	  1.00&	 2.41E-03&	 6.87E-04&	 7.13E-04\\
8&	 3.02E-03&	  1.00&	 1.87E-03&	 5.39E-04&	 6.09E-04\\
\bottomrule
\end{tabular}
}
\caption{Results for the Heat problem as the number of sub-intervals $M$ is varied. Here $K = 2, \Delta t = 0.10$.}
\label{tab:problem_Heat_R_10_J_2_M_vary}
\end{table}

\section{Conclusions and Future Directions}
In this work we develop an \emph{a posteriori}  error estimate in a quantity of interest obtained by solving an ODE using the spectral deferred correction method.   This is done by the use of a nodally equivalent finite element method to the SDC method.  The optimal order of this finite element method was determined based on the order of the SDC approximation.  Also, a method for determining the finite element approximation was developed based solely on the SDC approximation at the nodes.    There are many parameters that affect the method's convergence and error.  The error representation formula decomposes the error  into components that correspond to each of the parameters in the SDC method.  This informs the user how best to adjust the parameters to reduce the error in an approximation.  

The SDC method is also well-suited for a parallel implementation~\cite{speck:2015,Speck2018}.  In our future research we intend to develop an error representation formula for the parallel version of the SDC method.

\section*{Funding}
J. Chaudhry's  work is supported by the NSF-DMS 1720402.

\bibliographystyle{plain}
\bibliography{SDCError_bib}	
	
\end{document}